\documentclass{amsart}

\usepackage{amssymb,stmaryrd,amsxtra}

\newtheorem{thm}{Theorem}[section]
\newtheorem{lemma}[thm]{Lemma}
\newtheorem{prop}[thm]{Proposition}
\newtheorem{cor}[thm]{Corollary}
\newtheorem{fact}[thm]{Fact}

\theoremstyle{definition}
\newtheorem{df}[thm]{Definition}
\newtheorem{ex}[thm]{Example}
\newtheorem{question}[thm]{Question}
\newtheorem{rmk}[thm]{Remark}

\makeatletter
\def\dotminussym#1#2{%
  \setbox0=\hbox{$\m@th#1-$}%
  \kern.5\wd0%
  \hbox to 0pt{\hss\hbox{$\m@th#1-$}\hss}%
  \raise.6\ht0\hbox to 0pt{\hss$\m@th#1.$\hss}%
  \kern.5\wd0}
\makeatother

\newcommand{\dotminus}{\mathbin{\mathpalette\dotminussym{}}}

\begin{document}

\title{Pseudofinite and pseudocompact metric structures}
\author{Isaac Goldbring}
\thanks{Goldbring's work was partially supported by NSF grant DMS-1007144.}
\address{University of California, Los Angeles, Department of Mathematics\\ 520 Portola Plaza, Box 951555\\ Los Angeles, CA 90095-1555, USA}
\email{isaac@math.ucla.edu}
\urladdr{http://www.math.ucla.edu/~isaac}
\author[Vinicius C.\,L.]{Vinicius Cif\'u Lopes}
\address{Universidade Federal do ABC\\ Av.~dos Estados, 5001\\ Santo Andr\'e, SP 09210-971, Brazil}
\email{vinicius@ufabc.edu.br, vinicius@alumni.illinois.edu}
\urladdr{http://hostel.ufabc.edu.br/~vinicius}

\begin{abstract}
We initiate the study of pseudofiniteness in continuous logic.  We introduce a related concept, namely that of pseudocompactness, and investigate the relationship between the two concepts.  We establish some basic properties of pseudofiniteness and pseudocompactness and provide many examples.  We also investigate the injective-surjective phenomenon for definable endofunctions in pseudofinite structures.  
\end{abstract}

\maketitle

\section{Introduction}\label{sec-intro}

Pseudofiniteness is a well-known, interesting, and useful notion in classical logic; see e.g.~\cite{Ax,Felgner,Hru,Vaa}. Our goal is to introduce a concept in continuous logic in the setting of \cite{BBHU,BU} which corresponds as much as possible to the classical one.

However, in view of the lack of actual negations in the formal language, two different notions arise in our study. We choose to define them after the names of \emph{pseudofiniteness} and \emph{strong pseudofiniteness;} each resembles a different aspect of the classical notion.

We also introduce the related concept of \emph{pseudocompactness} (and a corresponding stronger version), because compact structures in many cases appear to be the right counterpart in continuous logic of finite structures in classical logic: they are saturated, have no proper ultrapowers, and indeed are totally categorical. (We dedicate an appendix to this fact.) Although we show that pseudocompactness is equivalent to pseudofiniteness in many cases, the stronger versions are distinct in an essential way, and thus we end up considering three different concepts.

We present detailed proofs for statements which correspond to trivial or well-known properties of classical pseudofiniteness, in order to highlight the nuances of continuous logic.

Section~\ref{sec-def} defines those four properties, and proves some elementary results about them.

In Section~\ref{sec-class}, we show that, in the case of classical languages and structures, the four notions coincide with the original one; this seems to require unusual attention.

Section~\ref{sec-ex} introduces some basic examples and questions.

Section~\ref{sec-relation} proves the equivalence of pseudofiniteness and pseudocompactness for relational languages and argues considerably in favor of a conjecture of general equivalence. Currently, we need to recourse to ``almost structures'', which satisfy a weaker clause of modulus of continuity for each nonlogical symbol.

In Section~\ref{sec-injsur}, we discuss the injectivity-surjectivity of definable endofunctions, that is, whether injective definable functions of the form $X\to X$ are surjective, and conversely. In classical logic, that property is a straightforward consequence of (and one main source of interest in) pseudofiniteness, although it holds independently of the latter as well. In continuous logic, work is more complex and requires strong pseudofiniteness and a strong assumption on the definable function; it also helps to distinguish between pseudofiniteness and strong pseudofiniteness.

We assume that the reader is familiar with the basics of continuous logic; otherwise, they may consult the paper \cite{BU} or the survey \cite{BBHU}.  There is one nuance of continuous logic that we use throughout the paper, which we mention here.  Since continuous logic lacks negation, one cannot, \emph{a priori,} express implications.  However, there is a trick that one can use to get around this problem:

\begin{fact}\cite[Prop.~7.14]{BBHU}
Suppose that $L$ is a $1$-bounded continuous signature, $M$ is an $\omega$-saturated $L$-structure, and $\varphi(x)$ and $\psi(x)$ are two $L$-formulae, where $x$ is an $n$-tuple of variables.  Then the following are equivalent:
\begin{enumerate}
\item For all $a\in M^n$, if $\varphi^M(a)=0$, then $\psi^M(a)=0$;
\item There is an increasing, continuous function $\alpha\colon[0,1]\to [0,1]$ satisfying $\alpha(0)=0$ so that, for all $a\in M^n$, we have $\psi^M(a)\leq \alpha(\varphi^M(a))$.
\end{enumerate}
\end{fact}

The import of this fact is that the second condition is indeed expressible by the condition $\sup_x (\psi(x)\dotminus \alpha(\varphi(x)))=0$.  

\section{Definitions and basic properties}\label{sec-def}

Until further notice, $L$ is a $1$-bounded metric signature (assumed to be one-sorted for simplicity), and $M$ is an $L$-structure.

\begin{df}
We say that $M$ is \emph{pseudofinite} (resp.\ \emph{pseudocompact}) if $\sigma^M=0$ for any $L$-sentence $\sigma$ such that $\sigma^A=0$ for all finite (resp.\ compact) $L$-structures $A$.
\end{df}

\begin{rmk}
Unfortunately, the term ``pseudocompact'' is already in use in the topology literature:  A space $X$ is said to be pseudocompact if any continuous function $X\to \mathbb{R}$ is bounded.  There is no relationship between our notion and the previous notion and thus this should not be a source of confusion.
\end{rmk}

\begin{rmk}
If $L$ is a many-sorted language, then an $L$-structure is said to be finite (resp.\ compact) if the underlying universe of each sort is finite (resp.\ compact), regardless of the number of sorts.  Then one defines pseudofinite and pseudocompact $L$-structures exactly as in the above definition.  
\end{rmk}

Clearly pseudofinite structures are pseudocompact.

\begin{lemma}\label{core-equiv}
The following are equivalent:
\begin{enumerate}
\item $M$ is pseudofinite (resp.\ pseudocompact);
\item For any $L$-sentence $\sigma$, if $\sigma^M=0$, then for any $\epsilon>0$, there is a finite (resp.\ compact) $L$-structure $A$ such that $\sigma^A\leq \epsilon$;
\item There is a set $\{A_i \colon i\in I\}$ of finite (resp.\ compact) $L$-structures and an ultrafilter $\mathcal{U}$ on $I$ such that $M\equiv \prod_\mathcal{U} A_i$.
\end{enumerate}
\end{lemma}

As usual, (3) together with the Keisler--Shelah Theorem (see \cite{HI} for a proof in the context of the approximate semantics predating continuous logic), provides a purely algebraic and logic-free characterization of pseudofinite (resp.\ pseudocompact) structures, namely, as those which have some ultrapower isomorphic to some ultraproduct of finite (resp.\ compact) structures.

\begin{proof}
(1)$\Rightarrow$(2):  Suppose that (1) holds but (2) fails.  Then there is an $L$-sentence $\sigma$ and an $\epsilon>0$ such that $\sigma^M=0$ but $\sigma^A\geq \epsilon$ for all finite (resp.\ compact) $L$-structures $A$.  But then $(\epsilon\dotminus \sigma)^A=0$ for all such $A$, whence $\sigma^M\geq \epsilon$ by (1), which is a contradiction. Note for future reference that the converse is similar: Suppose that (2) holds and that $\sigma$ is such that $\sigma^A=0$ for all finite $L$-structures $A$, yet $\sigma^M=:r>0$. Consider $\epsilon\in(0,r)$; then $|\sigma^A-r|\leq \epsilon$ for some such $A$, whence $\sigma^A>0$, a contradiction.

(2)$\Rightarrow$(3):  Assume that (2) holds.  Let $T=\operatorname{Th}(M)$ and let $J$ be the collection of finite subsets of $T$.  For each $\Delta\in J$ and $k\in \mathbb{N}^{>0}$, let $A_{\Delta,k}$ be a finite (resp.\ compact) $L$-structure such that $A_{\Delta,k}\models \max \Delta\leq \frac{1}{k}$, and let \[ X_{\Delta,k}=\{(\Gamma,l)\in J\times \mathbb{N}^{>0} \colon A_{\Gamma,l}\models \max \Delta\leq \tfrac{1}{k}\}. \]  Note that $\{X_{\Delta,k} \colon (\Delta,k)\in J\times \mathbb{N}^{>0}\}$ has the finite intersection property:  given $\Delta_1,\ldots,\Delta_m\in J$ and $k_1,\ldots,k_m\in \mathbb{N}^{>0}$, we have that \[ (\Delta,k)\in X_{\Delta_1,k_1}\cap \ldots \cap X_{\Delta_m,k_m}, \] where $\Delta=\Delta_1\cup \cdots \cup \Delta_m$ and $k=\max(k_1,\ldots,k_m)$.  Let $\mathcal{U}$ be an ultrafilter on $J\times \mathbb{N}^{>0}$ extending $\{X_{\Delta,k} \colon (\Delta,k)\in J\times \mathbb{N}^{>0}\}$.  Set $N:=\prod_\mathcal{U} A_{\Delta,k}$.  We claim that $M\equiv N$.  To see this, suppose that $\sigma^M=0$ and take any $k>0$. Then $X_{\{\sigma\},k}\in\mathcal{U}$, so that $\sigma^N\leq \frac{1}{k}$.  Since $k>0$ is arbitrary, we have that $\sigma^N=0$.

(3)$\Rightarrow$(1) is clear.
\end{proof}

We have a preservation result, whose proof is similar to that for classical logic:

\begin{lemma}
Any ultraproduct of pseudofinite $L$-structures is pseudofinite. If $M$ is pseudofinite, so is any L-structure elementarily equivalent to $M$, any reduct of $M$ to a sublanguage of $L$, and any expansion of $M$ by constants. The analogous statements for pseudocompactness also hold.
\end{lemma}

\begin{proof}
The only statement whose proof is not identical to that in classical logic is the one about expansion by constants, yet the procedure is similar: Given a sentence $\sigma$ in the expanded language, replace the new constant symbols by fresh variables $x$, thus obtaining an $L$-formula $\varphi(x)$. Assume that $\sigma^B=0$ for every finite structure $B$ in the expanded language. Any such $B$ equals $(A,a)$ for some $L$-structure $A$ and some sequence of parameters $a$ in $A$. Note then $[\sup_x\varphi(x)]^A=0$ for any finite $L$-structure $A$, since the choice of $a$ is arbitrary, hence $[\sup_x\varphi(x)]^M=0$. Now let $M'$ be any expansion of $M$ with interpretations for the new constant symbols: we have $\sigma^{M'}=0$.
\end{proof}

In the classical analog of Lemma \ref{core-equiv}, item (2) is replaced by the following statement: Whenever ${M}\models \sigma$, then ${A}\models \sigma$ for some finite structure ${A}$.  In other words, the approximate truth appearing in item (2) above is replaced by actual truth, and equivalence holds because of using the negation connective.  This motivates us to make the following definition:

\begin{df}
We say that $M$ is \emph{strongly pseudofinite} (resp.\ \emph{strongly pseudocompact}) if for any $L$-sentence $\sigma$ such that $\sigma^M=0$, there is a finite (resp.\ compact) $L$-structure $A$ such that $\sigma^A=0$.
\end{df}

In other words, $M$ is an addherence point of finite (resp.\ compact) structures, that is, for every $L$-sentence $\sigma$ such that $\sigma^M=0$, the ``neighborhood'' $\{N\colon\sigma^N=0\}$ contains a finite structure.

Observe that a finitely axiomatizable theory with no finite models cannot be strongly pseudofinite.

Of course, strongly pseudofinite structures are strongly pseudocompact, and Lemma~\ref{core-equiv} yields that each strong concept implies its corresponding plain version. See Examples~\ref{apaa} and~\ref{circle} for proved distinctions.

The following preservation lemma is almost as bold as the previous one: our examples (especially~\ref{circle}) will show that strong pseudofiniteness is not preserved under ultraproducts.

\begin{lemma}
If $M$ is strongly pseudofinite, so is any L-structure elementarily equivalent to $M$, any reduct of $M$ to a sublanguage of $L$, and any expansion of $M$ by constants. The analogous statements for pseudocompactness also hold.
\end{lemma}

\begin{proof}
Again, the only proof not identical to that in classical logic is the one about expansion by constants, yet it is similar: Given any sentence in the expanded language, replace the new constant symbols by fresh variables and quantify over each of them using the $\inf$ quantifier, then recall that $\inf$ quantifiers are actually realized in finite or compact structures.
\end{proof}

Next, we prove that pseudofiniteness and pseudocompactness are preserved when adding imaginaries.  We follow the approach to imaginaries (as well as the notation) from \cite[Sec.~5]{BBHU}, which we briefly recall here.  Given a family $(\varphi_n(x,y_n)\colon n\in\mathbb{N})$ of $L$-formulae, we consider the definable predicate $\psi(x,Y):=\mathcal{F}\lim \varphi_n$, which is the forced limit of these formulae.  For each such definable predicate $\psi(x,Y)$, we add an imaginary sort $(S_\psi,d_\psi)$ for canonical parameters of instances of $\psi$; given a tuple $b$ of sort $Y$ in some $L$-structure $M$, we let $[b]_\psi$ denote its image in $M^{\textup{eq}}$ of sort $S_\psi$.  We also add a predicate symbol $P_\psi(x,z)$ (where $z$ is of sort $S_\psi$), whose interpretation will satisfy $P_\psi(x,[b]_\psi)=\psi(x,b)$, and predicate symbols $\gamma_{\psi,n}$ which are approximations to the graph of the quotient map between tuples and their equivalence classes in $S_\psi$.  We let $L^{\textup{eq}}$ denote the resulting language and let $T_0^{\textup{eq}}$ denote the $L^{\textup{eq}}$-theory axiomatizing the properties of the new symbols (so the models of $T_0^{\textup{eq}}$ are precisely the $\textup{eq}$-expansions of $L$-structures).  \emph{We emphasize that this approach to imaginaries is independent of any ambient $L$-structure.}

Given a finite or countable tuple ${b}$ and $m>0$, we set ${b}|_m$ to be the truncation of $b$ to the first $m$ elements; we also refer to ${b}|_m$ simply as a \emph{truncation} of ${b}$.

\begin{lemma}
Let $\varphi(u,v_1,\ldots,v_n)$ be an $L^{{\textup{eq}}}$-formula, where $u$ is a tuple of variables from $L$ and $v_1,\ldots,v_n$ are variables from imaginary sorts.  Given $\epsilon>0$, there is an $L$-formula $\varphi'(u,v^1,\ldots,v^n)$ such that for all $L$-structures $M$ and $[b_i]\in M^{{\textup{eq}}}$ of the same sort as $v_i$, there are truncations $b_i'$ such that \[ M^{{\textup{eq}}}\models \sup_u |\varphi(u,[b_1],\ldots,[b_n])-\varphi'(u,{b}_1',\ldots,{b}_n')|\leq \epsilon. \]  In particular, for any $L^{{\textup{eq}}}$-sentence $\sigma$ and any $\epsilon>0$, there is an $L$-sentence $\sigma'$ such that $T_0^{{\textup{eq}}}\models |\sigma-\sigma'|\leq \epsilon$.
\end{lemma} 

\begin{proof}
The proof is by induction on the complexity of $\varphi$.  First suppose that $\varphi$ is atomic.  Without loss of generality, we may suppose that $\varphi$ is not an $L$-formula.  Thus, there is a definable predicate $\psi({x},Y)=\mathcal{F}\lim(\varphi_n(x,y_n))$ such that $\varphi$ is of the form: \[ P_\psi({x},z),\ d_\psi(z,z^*),\ \text{or}\ \gamma_{\varphi_n,\psi}({y_n},z). \]  In the first case, choose $N$ such that $2^{-N}\leq \epsilon$.  Then \[ M^{{\textup{eq}}}\models \sup_{{x}}|P_\psi({x},[b])-\varphi_N({x},{b|_N})|\leq \epsilon. \]  In the second case, choose $N$ so that $2^{-N+1}\leq \epsilon$.  Then \[ M^{{\textup{eq}}}\models\Bigl|d([b],[b^*])-\sup_{{x}}|\varphi_N({x},{b|_N})-\varphi_N({x},{b^*|_N})|\Bigr|\leq \epsilon. \]  Finally, for the third case, choose $N$ so that $2^{-N}\leq \epsilon$.  Then \[ M^{{\textup{eq}}}\models \sup_{{y_n}}\Bigl|\gamma_{\varphi_n}({y_n},[b])-\sup_{{x}}|\varphi_n({x},{y_n})-\varphi_N({x},{b|_N})\Bigr|\leq \epsilon. \] The connective step of the proof follows immediately from uniform continuity.  There are two quantifier cases to consider.  First suppose that $\varphi(u,v_1,\ldots,v_n)=\inf_w \chi(u,w,v_1,\ldots,v_n)$, where $w$ is a variable of $L$.  Let $\chi'(u,w,v^1,\ldots,v^n)$ be as in the conclusion of the lemma for $\chi$ and $\epsilon$.  Then set $\varphi'(u,v^1,\ldots,v^n):=\inf_w \chi'(u,w,v^1,\ldots,v^n)$.  Now assume $\varphi(u,v_1,\ldots,v_n)=\inf_w \chi(u,v_1,\ldots,v_n,w)$, where $w$ is an imaginary variable.  Let $\chi'(u,v^1,\ldots,v^n,w')$ be an $L$-formula satisfying the conclusion of the lemma for $\chi$ and $\epsilon$.  Then set $\varphi'(u,v^1,\ldots,v^n):=\inf_{w'}\chi'(u,v^1,\ldots,v^n,w')$.
\end{proof}

In order to make the desired preservation result true, we need to make a slight modification to the eq-construction for finite structures.  Suppose that $A$ is a finite $L$-structure and $\psi({x},Y)$ is a definable predicate with $Y$ a countable tuple of parameter variables. Observe then that there is a \emph{finitary} definable predicate $\widetilde{\psi}({x},{y})$ (that is, ${y}$ is a finite tuple) such that $\psi$ and $\widetilde{\psi}$ are logically equivalent.  Consequently, there is no need to add the sort $S_\psi$ if one adds the sort $S_{\widetilde{\psi}}$.  Thus, we insist that the eq-construction for finite structures only add sorts for canonical parameters of finitary imaginaries.  In this way, $A^{{\textup{eq}}}$ is once again a finite structure.  Observe that if $A$ is a compact structure, then $A^{{\textup{eq}}}$ is also a compact structure.

\begin{prop}
If $M$ is pseudofinite (resp.\ pseudocompact), then so is $M^{{\textup{eq}}}$.
\end{prop}

\begin{proof}
Suppose that $M$ is pseudofinite and $M^{{\textup{eq}}}\models \sigma=0$.  Fix $\epsilon>0$ and let $\sigma'$ be as in the lemma. Then $M\models \sigma'\leq \epsilon$. Then there is a finite (resp.\ compact) $L$-structure $A$ such that $A\models \sigma'\leq 2\epsilon$. Since also $A^{{\textup{eq}}}\models |\sigma-\sigma'|\leq \epsilon$, we have $A^{{\textup{eq}}}\models \sigma\leq 3\epsilon$.
\end{proof}

\begin{question}
Are the notions of strong pseudofiniteness and strong pseudocompactness preserved when adding imaginaries? More generally, how natural is strong pseudofiniteness in continuous logic, which relies heavily on approximations?
\end{question}

The next two sections present a list of many examples of pseudofinite and pseudocompact metric structures.

\section{The case of a classical structure}\label{sec-class}

In this section, we let $L$ denote a signature for classical logic.  In order for us to treat $L$ as a signature for continuous logic, we must also specify a modulus of uniform continuity for each symbol.  For each function and predicate symbol, we use $\Delta(\epsilon):=\epsilon$ as our modulus of uniform continuity.  Note that any classical $L$-structure, when equipped with the discrete metric, is a metric $L$-structure.

For the rest of this section, let $M$ denote a classical $L$-structure.  If $M$ is pseudofinite in the classical sense, then we will say that $M$ is \emph{classically pseudofinite}.  Thus, when we say that $M$ is pseudofinite, we are considering $M$ as a metric $L$-structure.  It is not immediately clear that $M$ is classically pseudofinite if and only if $M$ is pseudofinite as there are metric $L$-structures that are not the result of viewing classical $L$-structures as metric structures.  Nonetheless, in this section, we will prove:

\begin{thm}\label{classical}
Given a classical $L$-structure, the five notions \emph{classically pseudofinite,} \emph{pseudofinite,} \emph{strongly pseudofinite,} \emph{pseudocompact,} and \emph{strongly pseudocompact} coincide.
\end{thm}

Towards proving this result, note that, since all five notions are invariant under elementary equivalence, we may replace $M$ with an ultrapower, thus reducing to the case that $M$ is $\omega$-saturated (as a metric structure).

One proves the next lemma by induction on the complexity of formulae:

\begin{lemma}\label{formula}
For any continuous $L$-formula $\varphi(x)$, there is finite $R_\varphi\subseteq [0,1]$ such that, for any \emph{classical} $L$-structure $A$ and any tuple $a$ from $A$, one has $\varphi^A(a)\in R_{\varphi}$.
\end{lemma}

We should remark that it is essential that we assumed that the structures in the previous lemma are classical.  Indeed, if $\varphi(x,y)$ is $d(x,y)$, then as $A$ ranges over all metric $L$-structures, $\varphi^A$ can take on any value in $[0,1]$.

\begin{lemma}
Suppose that $A$ and $B$ are classical $L$-structures and that $A\equiv B$ as classical structures.  Then $A\equiv B$ as metric structures.
\end{lemma}

\begin{proof}
By passing to a common elementary extension, we reduce to the case that $A$ is an elementary substrcture of $B$ (as classical structures).  In this case, the lemma then follows from the previous lemma and \cite[9.21]{BBHU}, which is uniform over all classical structures. Indeed, given a continuous $L$-formula $\varphi(x)$ and $r\in R_\varphi$, there is a classical $L$-formula $\psi(x)$ such that, for any classical $L$-structure $A$ and for any suitable tuple $a$ from $A$, we have $\varphi^A(a)=r\Leftrightarrow A\models\psi(a)$.
\end{proof}

\begin{cor}
If $M$ is classically pseudofinite, then $M$ is strongly pseudofinite.
\end{cor}

\begin{proof}
Let there be finite classical $L$-structures $(A_i\colon i\in I)$ and an ultrafilter $\mathcal{U}$ on $I$ such that $M\equiv \prod_{\mathcal{U}} A_i$ (a classical ultraproduct).  By the previous lemma, $M\equiv \prod_{\mathcal{U}} A_i$ as metric structures.  Consequently, if $\sigma$ is a continuous $L$-sentence such that $\sigma^M=0$, then $\lim_{\mathcal{U}} \sigma^{A_i}=0$.  By Lemma \ref{formula}, $\sigma^{A_i}=0$ for some (actually, almost all) $A_i$.
\end{proof}

Note that, for all $a,b\in M$, we have $d(a,b)\leq \frac{1}{2}\Rightarrow d(a,b)=0$.  Thus, since $M$ is $\omega$-saturated, there is an increasing, continuous $\alpha\colon[0,1]\to [0,1]$ such that $\alpha(0)=0$ and $d(a,b)\leq \alpha(d(a,b)\dotminus \frac{1}{2})$ for all $a,b\in M$.  We fix this $\alpha$ for the rest of the section.
\begin{lemma}
If $M$ is strongly pseudocompact, then $M$ is strongly pseudofinite.
\end{lemma}  

\begin{proof}
Suppose $\sigma^M=0$.  Then \[ M\models \max\Bigl(\sigma,\sup_{x,y}(d(x,y)\dotminus \alpha(d(x,y)\dotminus\tfrac{1}{2}))\Bigr)=0. \] Thus the displayed sentence has value $0$ in some compact $L$-structure $A$.  However, the second ``conjunct'' forces $A$ to be discrete, whence finite.
\end{proof}

In order to finish the proof of Theorem \ref{classical}, it remains to prove that if $M$ is pseudocompact, then it is classically pseudofinite.

Towards this end, given a classical $L$-formula $\varphi(x)$, we define its \emph{continuous transform} $\widetilde{\varphi}(x)$ by recursion as follows: \begin{itemize}
	\item if $\varphi(x)$ is $t_1(x)=t_2(x)$, then $\widetilde{\varphi}(x)$ is $d(t_1(x),t_2(x))$;
	\item if $\varphi(x)$ is $P(t_1(x),\ldots,t_n(x))$, then $\widetilde{\varphi}(x)$ is $\varphi(x)$;
	\item $[\neg\varphi]\sptilde$ is $1\dotminus\widetilde{\varphi}$;
	\item $[\varphi\wedge\psi]\sptilde$ is $\max(\widetilde{\varphi},\widetilde{\psi})$;
	\item $[\exists y\,\psi(x,y)]\sptilde$ is $\inf_y\widetilde{\psi}(x,y)$.
\end{itemize} It is easy to see that for any classical $L$-structure $B$, any classical $L$-formula $\varphi(x)$, and any $b$ from $B$, we have $B\models \varphi(b) \Rightarrow \widetilde{\varphi}(b)^B=0$.

Now suppose that $M$ is pseudocompact and that $\sigma$ is a classical $L$-sentence such that $M\models \sigma$.  We will find a finite classical $L$-structure $B$ such that $B\models \sigma$.

For every predicate symbol $P$ in $\sigma$ and every suitable tuple $a$ from $M$, we have $P^M(a)\leq \frac{1}{2}\Rightarrow P^M(a)=0.$  Thus, there is an increasing, continuous $\alpha_P\colon[0,1]\to[0,1]$ with $\alpha_P(0)=0$ such that $\sigma_P^M=0$, where \[ \sigma_P:=\sup_{x}(P(x)\dotminus \alpha_P(P(x)\dotminus \tfrac{1}{2})). \]  Similarly, letting $\tau:=\sup_{x,y}(d(x,y)\dotminus \alpha(d(x,y)\dotminus \frac{1}{2}))$, we also have that $\tau^M=0$.  Let $\sigma':=\max(\widetilde{\sigma},(\sigma_P)_P,\tau)$, where the $P$'s range over the predicate symbols appearing in $\sigma$.

Since $M$ is pseudocompact, there is a compact structure $A$ such that $(\sigma')^A\leq \frac{1}{4}$.  Define a binary relation $E$ on $A$ by $E(x,y):\Leftrightarrow d(x,y)\leq \frac{1}{4}$.  Clearly $E$ is reflexive and symmetric.  However, $E$ is also transitive: indeed, if $d(x,y),d(y,z)\leq \frac{1}{4}$, then $d(x,z)\leq \frac{1}{2}$; since $\tau^A\leq \frac{1}{4}$, we get $d(x,z)\leq \frac{1}{4}$.  We now define the desired finite classical $L$-structure as follows.  We define the underlying set of $B$ to be the set of $E$-equivalence classes of $A$.  Since $A$ has a finite $\frac{1}{4}$-net, we have that $B$ is finite.  Now we need to interpret the structure.  For our purposes, it is only relevant how we define the symbols appearing in $\sigma$. Given a predicate symbol $P$ appearing in $\sigma$, we declare $[a]\in P^B$ if and only if $P^A(a)\leq \frac{1}{2}$.  To see that we can do so, we need to check that if $d(a_i,a_i')\leq \frac{1}{4}$ for each $i$, then $P^A(a)\leq \frac{1}{2}\Leftrightarrow P^A(a')\leq \frac{1}{2}$.  However, since $A\models \sigma_P\leq \frac{1}{4}$, then $P^A(a)\leq \frac{1}{2}$ implies $P^A(a)\leq \frac{1}{4}$.  Since we are using the identity as modulus of uniform continuity for all of our symbols, we get that $|P^A(a)-P^A(a')|\leq \frac{1}{4}$, giving the desired result.  Similarly, for every function symbol $f$ appearing in $\sigma$, we define $f^B([a]):=[f^A(a)]$.  Again, we can do so because $[a]=[a']$ implies $|f^A(a)-f^A(a')|\leq \frac{1}{4}$ as we are using the identity modulus of uniform continuity.

Let $L'$ be the reduct of $L$ that only contains the symbols of $\sigma$.  It is easy to prove, by induction on the complexity of formulae, the following:

\begin{lemma}
Suppose that $\varphi(x)$ is a classical $L'$-formula.  Then for any suitable tuple $a$ in $A$, we have: \begin{enumerate}
	\item If $\widetilde{\varphi}^A(a)\leq \frac{1}{4}$, then $B\models \varphi([a])$;
	\item If $\widetilde{\varphi}^A(a)\geq \frac{3}{4}$, then $B\models \neg \varphi([a])$.
\end{enumerate} 
\end{lemma}

In proving the above lemma, one needs to use that inf quantifiers are realized in compact structures.  By the lemma, since $\widetilde{\sigma}^A\leq \frac{1}{4}$, we have $B\models \sigma$.
This discussion proves:

\begin{lemma}
If $M$ is pseudocompact, then $M$ is classically pseudofinite.
\end{lemma}

This finishes the proof of Theorem \ref{classical}.  We should remark that the above discussion is unusual in the sense that when one generalizes a notion to continuous logic, it is often immediate that the notion agrees with the classical notion on classical structures.  For example, if $T$ is a classical theory, then it is immediate to see that $T$ is stable as a classical theory if and only if $T$ is stable as a continuous theory.  It is interesting to note that pseudofiniteness appears to be the first notion where some work is required to show that the notions agree on classical structures.

\section{Examples and questions}\label{sec-ex}

\begin{ex}
Let $L$ be the signature naturally used for closed unit balls of inner product spaces (see \cite[Ex.~2.1]{BBHU}). For $n\geq 1$, let $B_n$ denote the closed unit ball of $\mathbb{R}^n$, viewed naturally as an $L$-structure.  Let $\mathcal{U}$ be any nonprincipal ultrafilter on $\mathbb{N}$ and let $H=\prod_\mathcal{U} B_n$.  Clearly $H$ is the closed unit ball of an infinite-dimensional Hilbert space.  By completeness, we see that \emph{any closed unit ball of a Hilbert space is pseudocompact.}

We will be able to show that $H$ is not strongly pseudofinite; the function $H\to H$, $x\mapsto\tfrac{1}{2}x$, is injective but not surjective.

There are some interesting pseudocompact expansions of $H$.  First, consider $(H,P)=\prod_\mathcal{U}(B_{n},P_{n})$, where $P_{n}\colon\mathbb{R}^{n}\to \mathbb{R}^{n}$ is a projection operator onto an $\lfloor n/2\rfloor$-dimensional subspace of $\mathbb{R}^{n}$.  Then $(H,P)$ is a pseudocompact structure and is a model of the theory of \emph{beautiful pairs} of Hilbert spaces, namely an infinite-dimensional Hilbert space equipped with a projection with infinite-dimensional image and infinite-dimensional orthogonal complement.  (See \cite{BV}.)

Next, let $\{z_i \colon i<\omega\}$ be a countable dense subset of $\{z\in \mathbb{C} \colon |z|=1\}$.  Let $(H,U)=\prod_\mathcal{U} (B_n,U_n)$, where $U_n\colon\mathbb{C}^n\to \mathbb{C}^n$ is a unitary operator with eigenvalues $\{z_0,\ldots,z_n\}$.  Then for any $m\geq n$, we have \[ (B_m,U_m)\models \inf_x \max\bigl(|\langle x,x\rangle -1|,\|U(x)-z_n x\|\bigr)=0. \]  It follows from the work in \cite{BUZ} that $(H,U)$ is a model of the theory of Hilbert spaces equipped with a generic automorphism.
\end{ex}

\begin{ex}\label{apaa}
Let $L$ be the signature for probability structures (see \cite[Sec.~16]{BBHU}).  Let $\mathcal{B}_n$ be the probability structure with event algebra $2^n$ and with the counting measure $\mu_n$.  Let $\mathcal{U}$ be a nonprincipal ultrafilter on $\mathbb{N}$ and let $\mathcal{B}=\prod_\mathcal{U} \mathcal{B}_n$.  We claim that $\mathcal{B}$ is an atomless probability structure; since the theory of atomless probability structures is complete, it follows that \emph{any atomless probability structure is pseudofinite}.  Towards this end, suppose that $x=[(x_n)]\in \mathcal{B}$ is such that $\mu(x)>0$.  Set $m_n:=|x_n|$; then $m_n>1$ for almost all $n$, for otherwise $\mu_n(x_n)\leq \frac{1}{n}$ and hence $\mu(x)=\lim_{\mathcal{U}}\mu_n(x_n)=0$.  For such $n$, let $y_n\subseteq x_n$ be such that $|y_n|=\frac{1}{2}m_n$ if $m_n$ is even or $|y_n|=\frac{1}{2}(m_n-1)$ if $m_n$ is odd.  Then $\mathcal{B}_n\models \inf_y|\mu(x_n\cap y)-\mu(x_n\cap y^c)|\leq \frac{1}{n}$ for almost all $n$, whence $\mathcal{B}\models \inf_y|\mu(x\cap y)-\mu(x\cap y^c)|=0$.  It follows that $\mathcal{B}$ is atomless.  Observe that, since the (complete) theory of atomless probability structures is finitely axiomatizable, it cannot be strongly pseudofinite.

An extension of this idea shows that any atomless probability structure equipped with a \emph{generic (or aperiodic) automorphism} is also pseudofinite (but of course not strongly pseudofinite).  Indeed, the theory of atomless probability structures equipped with a generic automorphism (denoted $\text{APAA}$ in \cite[Sec.~18]{BBHU}) is axiomatized (in the language of probability structures expanded by a unary function symbol $\tau$) by the axioms of atomless probability structures equipped with an automorphism together with, for each $n\geq1$, the axiom \[ \inf_e\max\bigl(|\tfrac{1}{n}-\mu(e)|,\mu(e\cap \tau(e)),\ldots,\mu(e\cap \tau^{n-1}(e))\bigr)=0. \]  We consider the probability structures $\mathcal{B}_m$ from above and we equip them with the automorphisms $\tau_m$ induced by the point map $x\mapsto x+1\mod m$.  Fix $n\geq 1$ and suppose that $m>n$.  Choose $k\in \{1,\ldots,m\}$ maximal with respect to $\frac{k-1}{m-1}\leq \frac{1}{n}$.  Let $e=\{1,1+n,\ldots,1+(k-1)n\}\in \mathcal{B}_m$.  Observe that $|e|=k$ so \[ |\mu(e)-\tfrac{1}{n}|\leq |\tfrac{k}{m}-\tfrac{k}{m-1}|+|\tfrac{k}{m-1}-\tfrac{k-1}{m-1}|+|\tfrac{k-1}{m-1}-\tfrac{1}{n}|\leq \tfrac{3}{m-1}. \] Furthemore, observe that for $i\in \{0,1,\ldots,n-1\}$ we have $\tau_m^i(e)\cap e\subseteq \{1\}$, so $\mu(\tau_m^i(e)\cap e)\leq \frac{1}{m}$ for each such $i$.  Consequently, \[ (\mathcal{B}_m,\tau_m)\models \inf_e\max\bigl(|\tfrac{1}{n}-\mu(e)|,\mu(e\cap \tau(e)),\ldots,\mu(e\cap \tau^{n-1}(e))\bigr)\leq \tfrac{3}{m-1} \] for each $m\geq n$.  Let $\tau_\infty=\lim_{\mathcal{U}}\tau_m$.  It follows that \[ (\mathcal{B},\tau_\infty)\models \inf_e\max\bigl(|\tfrac{1}{n}-\mu(e)|,\mu(e\cap \tau(e)),\ldots,\mu(e\cap \tau^{n-1}(e))\bigr)=0. \]  Since $n\geq 1$ was arbitrary, we have that $(\mathcal{B},\tau_\infty)\models \text{APAA}$.
\end{ex}

\begin{ex}
We can generalize the previous example as follows:  Let $L$ be a countable classical signature.  We claim that if ${M}$ is a pseudofinite ${L}$-structure, then any of its \emph{Keisler randomizations} are pseudofinite metric structures, whence pseudofiniteness is a robust model-theoretic notion in the sense that it is preserved under randomizations.  (See \cite{BK} for information on the notion of Keisler randomization.  The reason this example generalizes the previous example is that the theory of atomless probability algebras is just the theory of the Keisler randomization of a two-element set.)  To see this, suppose that ${M}\equiv \prod_{\mathcal{U}}{M}_n$, where $({M}_n\colon \ n\in \mathbb{N})$ is a family of finite $L$-structures and $\mathcal{U}$ is a nonprincipal ultrafilter on $\mathbb{N}$; observe that since $L$ is countable, we may always find such a countable family of finite structures.  We consider the probability structures $\mathcal{B}_n$ from the previous example and we let $\mathcal{K}_n:=M_n^{\{1,\ldots,n\}}$.  We claim that $(\mathcal{K},\mathcal{B}):=\prod_{\mathcal{U}}(\mathcal{K}_n,\mathcal{B}_n)\models T^R$, where $T:=\operatorname{Th}({M})$; since $T^R$ is complete, it suffices to prove this claim.  The Validity, Boolean, Distance, Fullness, Event, and Measure axioms are clearly ``true'' in each of the factor structures, and hence ``true'' in the ultraproduct.  The previous example already shows that the ultraproduct satisfies the Atomless Axiom.  It remains to verify the Transfer Axiom, namely, for every sentence $\sigma\in T$, we need $(\mathcal{K},\mathcal{B})\models d(\llbracket \sigma \rrbracket,\top)=0$.  Note that $d^{\mathcal{B}_n}(\llbracket \sigma\rrbracket, \top)=0$ if and only if ${M}_n\models \sigma$; since $\sigma\in T$, then ${M}_n\models \sigma$ for almost all $n$, whence $d^{\mathcal{B}}(\llbracket \sigma\rrbracket,\top)=\lim_{\mathcal{U}}d^{\mathcal{B}_n}(\llbracket \sigma \rrbracket,\top)=0$.  Note that this example provides us with our first examples of unstable pseudofinite theories (other than the classical ones).  Indeed, if $T$ is not stable, then $T^R$ is unstable (see~\cite{BK}).  Also, by a result of Ben Yaacov (see \cite{B2}), if $T$ is simple unstable, then $T^R$ is \emph{not} simple.  Thus, if $T$ is the theory of the random graph or the theory of pseudofinite fields, then $T^R$ is pseudofinite but not simple. 
\end{ex}

\begin{ex}
Let $(X,d)$ be a \emph{proper} metric space (that is, its closed balls are compact).  Fix a basepoint $p\in X$ and consider $(X,d,p)$ as a many-sorted structure in the natural way.  (See~\cite{Carlisle} for all the concepts.) Then the asymptotic cone of $(X,d,p)$ with respect to any nonprincipal ultrafilter on $\mathbb{N}$ will be pseudocompact, and hence pseudofinite by Theorem \ref{compfin} below.  Of particular interest is the case when $(X,d)$ is the Cayley graph of a finitely generated group $G$.  If $G$ is furthermore assumed to be hyperbolic, then this asymptotic cone will be an $\mathbb{R}$-tree.  If the hyperbolic group is \emph{nonelementary} (that is, it does not contain an infinite cyclic group of finite index), for example $G=\mathbb{F}_2$ the free group on two generators, then the $\mathbb{R}$-tree is \emph{richly branching}.  Since the theory of richly branching $\mathbb{R}$-trees is complete, we have that any richly branching $\mathbb{R}$-tree is pseudofinite.
\end{ex}

We end this section with a couple of questions.

\begin{question}
Let $\mathfrak{U}$ denote the \emph{bounded Urysohn space}, that is, the unique Polish metric space of diameter bounded by $1$ which is ultrahomogeneous and contains an isometric copy of every Polish metric space of diameter bounded by $1$.  Let $T_\mathfrak{U}$ denote the theory of $\mathfrak{U}$ in the metric signature containing only the distance symbol.  Is $\mathfrak{U}$ pseudofinite or pseudocompact?
\end{question}

The model theory of $\mathfrak{U}$ is quite well-understood; see the papers \cite{Usvy} and \cite{GE}, where it is shown that $T_\mathfrak{U}$ is complete, admits quantifier elimination, is $\aleph_0$-cat\-e\-gor\-i\-cal, and is rosy with respect to finitary imaginaries (but is not simple).  In some sense, the Urysohn space is the continuous analog of the random graph and so one might expect that the answer to the previous question is positive.  In \cite{Usvy}, an axiomatization for $T_\mathfrak{U}$ is given by writing down conditions in continuous logic describing a certain collection of ``extension axioms.''  Thus, the following lemma might prove useful.

\begin{lemma}
Suppose that $\{\gamma=0 \colon \gamma \in \Gamma\} \models \operatorname{Th}(M)$ for some collection $\Gamma$ of $L$-sentences, and that, for every $\gamma_1,\ldots,\gamma_n\in \Gamma$ and every $\epsilon>0$, there is a finite (resp.\ compact) $L$-structure $A$ such that $A\models \max(\gamma_1,\ldots,\gamma_n)\leq \epsilon$.  Then $M$ is pseudofinite (resp.\ pseudocompact).
\end{lemma}

\begin{proof}
Let $\sigma$ be an $L$-sentence such that $\sigma^M=0$.  Fix $\delta>0$.  Then by compactness, there is $\epsilon>0$ and $\gamma_1,\ldots,\gamma_n\in \Gamma$ such that \[ \{\max(\gamma_1,\ldots,\gamma_n)\leq \epsilon,\sigma\geq \delta\} \] is unsatisfiable.  Let $A$ be a finite (resp.\ compact) $L$-structure such that $A\models \max(\gamma_1\ldots,\gamma_n)\leq \epsilon$.  Then $\sigma^A<\delta$.  Since $\delta>0$ is arbitrary, we see that $M$ is pseudofinite (resp.\ pseudocompact).
\end{proof}

Thus, by the previous lemma, in order to show that $\mathfrak{U}$ is pseudofinite (resp.\ pseudocompact), one might just try to prove that any finite collection of extension axioms are approximately true in some finite (resp.\ compact) bounded metric space.  By Theorem~\ref{compfin} below, if one can prove that $\mathfrak{U}$ is pseudocompact, then it will follow that $\mathfrak{U}$ is pseudofinite.

\begin{question}
Is there an example of an ``essentially continuous'' strongly pseudofinite (resp.\ strongly pseudocompact) structure that is not finite (resp.\ compact)?
\end{question}

While the term ``essentially continuous'' is admittedly vague (although there have been attempts by others to make this notion precise), we use it to preclude discrete examples.  For example, any pseudofinite classical structure is strongly pseudofinite; we do not want such structures to constitute a positive answer to our question.  In the last section, we use some results concerning definable endofuctions in strongly pseudofinite structures to show that some pseudofinite structures are not strongly pseudofinite.  In fact, since continuous logic is an approximate logic, we conjecture that the answer to the above question might be negative.  It would even be interesting to settle this question under some extra set-theoretic hypotheses.

We should mention that there is a na\"ive attempt to construct a strongly pseudofinite structure, namely by using an ultraproduct relative to an \emph{$\omega_1$-complete ultrafilter}.  Recall that if $I$ is an index set and $\mathcal{U}$ is a nonprincipal ultrafilter on $I$, then $\mathcal{U}$ is $\omega_1$-complete if, whenever $E\subseteq \mathcal{U}$ is countable, then $\bigcap E\in \mathcal{U}$. If $I=\mathbb{N}$, then this is equivalent to the following:  whenever $I_1\supsetneq I_2\supsetneq I_3 \supsetneq \cdots$ is a decreasing sequence of subsets of $I$ such that $I_n\in \mathcal{U}$ for each $n$, then $\bigcap_{n\geq 1}I_n\not=\emptyset$ (see \cite[Prop.~4.3.3 and~4.3.4]{CK}). Suppose now that $(A_i\colon i\in I)$ is a family of finite (resp.\ compact) $L$-structures, $\mathcal{U}$ is an $\omega_1$-complete ultrafilter on $I$, and that $M$ is an $L$-structure satisfying $M\equiv \prod_{\mathcal{U}} A_i$.  Then $M$ is strongly pseudofinite (resp.\ strongly pseudocompact).  Indeed, suppose that $\sigma$ is a sentence such that $\sigma^M=0$.  Let $I_n:=\{i\in I\colon A_i\models \sigma^{A_i}\leq \frac{1}{n}\}\in \mathcal{U}$.  Then by assumption, there is $i\in \bigcap_{n\geq 1}I_n$, and $\sigma^{A_i}=0$ for this $i$. Unfortunately, $M$ will be finite (resp.\ compact).  Indeed, let \[ D_n=\{i\in I\colon |A_i|\geq n\} .\]  If each $D_n\in \mathcal{U}$, then the intersection of the $D_n$ is also in $\mathcal{U}$, contradicting that each $A_i$ is finite.  Consequently, $I\setminus D_n\in \mathcal{U}$ for some $n$, whence $M$ has size $<n$.  For the pseudocompact situation, given $\epsilon>0$, let \[ D_n:=\{i\in I \colon A_i\ \text{does not have an $\epsilon$-net of size}\ \leq n\}. \]  The above argument shows that $M$ has a finite $\epsilon$-net; since $\epsilon>0$ was arbitrary, $M$ is compact.

\section{Relationship between pseudofiniteness and pseudocompactness}\label{sec-relation}

In this section, $L$ denotes a $1$-bounded, one-sorted metric signature. We will prove:

\begin{thm}\label{compfin} Suppose that $L$ contains only predicate and constant symbols. Then the two notions \emph{pseudofinite} and \emph{pseudocompact} coincide.
\end{thm}

We will also obtain a similar result for languages with function symbols.

\emph{Henceforth, $M$ denotes a compact $L$-structure.  For each $m\geq 1$, let $X_m\subseteq M$ be a finite $\frac{1}{m}$-net for $M$.}

Suppose first that $L$ is relational.  We can then view each $X_m$ as a substructure of $M$.  We let $N:=\prod_\mathcal{U} X_m$, where $\mathcal{U}$ is some nonprincipal ultrafilter on $\mathbb{N}$.  We denote sequences from $\prod X_m$ as $(a^m)$ and write $[a^m]$ for the corresponding equivalence class in $N$.

If $c$ is a constant symbol in $L$, and if $c^M\in X_m$, we define $c^{X_m}=c^M$; if $c^M\notin X_m$, we choose $c^{X_m}$ so that $d(c^M,c^{X_m})<\frac{1}{m}$.  Now $X_m$ continues to be an $L$-structure, though not necessarily a substructure of $M$.

Suppose now that $L$ has function symbols.  If $f$ is a function symbol in $L$ and $a$ is a suitable tuple in $X_m$, define $f^{X_m}(a)$ to be an element of $X_m$ so that $d(f^{X_m}({a}),f^M({a}))<\frac{1}{m}$.  Observe that $X_m$ may not be an $L$-structure for the singular reason that it may not respect the modulus of uniform continuity for $f$ specified by $L$.  However, $X_m$ is an ``almost $L$-structure'' in the following sense:

\begin{df}
An \emph{almost $L$-structure} $X$ is defined as an $L$-structure, except that the clause of modulus of continuity for each function symbol $f$ is weakened thus: for sufficiently small $\epsilon$ and every $a,b$ from $X$, we require \[ d({a},{b})<\Delta_f(\epsilon)\Rightarrow d(f^{X}({a}),f^{X}({b}))\leq \epsilon ;\] and the clause of modulus of continuity for each predicate symbol is weakened analogously. (Every $L$-structure is an almost $L$-structure.)
\end{df}

Indeed, since the net $X_m$ is finite, there is $r_m>0$ such that $d(a,b)\geq r_m$ for each pair of distinct $a,b\in X_m$, in which case, when $\Delta_f(\epsilon)<r_m$, we have that $d({a},{b})<\Delta_f(\epsilon)\Rightarrow {a}={b}$.  (We are assuming here that $\lim_{\epsilon\to 0^+}\Delta_f(\epsilon)=0$, which is usually the case.)

Once again, let $N:=\prod_\mathcal{U} X_m$.  Observe that, although the $X_m$ are only almost $L$-structures, $N$ is an \emph{actual} $L$-structure.  Indeed, fix $\epsilon>0$ and suppose that $[a^m],[b^m]$ are tuples in $N$ such that $d([a^m],[b^m])<\Delta_f(\epsilon)$.  Then for almost all $m$, we have $d(a^m,b^m)<\Delta_f(\epsilon)$, whence $d(f^{X_m}(a^m),f^{X_m}(b^m))\leq \epsilon+\frac{2}{m}$.  Consequently, $d(f^N([a^m]),f^N([b^m]))\leq \epsilon$.

\begin{lemma}
For any $\epsilon>0$ and $L$-term $t(x)$, there is $K\in \mathbb{N}$ so that for all $m\geq K$ and all suitable $a$ from $X_m$, we have $d(t^M(a),t^{X_m}(a))<\epsilon$.  
\end{lemma}  

\begin{proof}
One proves this lemma by induction on the complexity of $t$.  When $t$ is a function symbol applied to variables or a constant symbol, then this follows from the definition of interpretation in $X_m$.  Now suppose that $t(x)=f(t_1(x),\ldots,t_k(x))$.  Choose $K'$ so that $\frac{1}{K'}<\Delta_f(\frac{\epsilon}{2})$, and then choose $K>\frac{2}{\epsilon}$ so that our claim holds for $t_1,\ldots,t_k$ with $\frac{1}{K'}$ in place of $\epsilon$.  Suppose that $m\geq K$ and $a$ lies in $X_m$.  Then $d(t^M_i(a),t^{X_m}_i(a))<\frac{1}{K'}$ for each $i$, whence \[ d\bigl(f^M(t_1^M(a),\ldots,t_k^M(a)),f^M(t_1^{X_m}(a),\ldots,t_k^{X_m}(a))\bigr)\leq \tfrac{\epsilon}{2} ,\] whence \[ d\bigl(f^M(t_1^M(a),\ldots,t_k^M(a)),f^{X_m}(t_1^{X_m}(a),\ldots,t_k^{X_m}(a))\bigr)\leq \tfrac{\epsilon}{2}+\tfrac{1}{m}<\epsilon .\]
\end{proof}

\begin{lemma}
For any $\epsilon>0$ and $L$-formula $\varphi(x)$, there is $K\in\mathbb{N}$ so that for all $m\geq K$ and all suitable $a$ from $X_m$, we have $|\varphi^M(a)-\varphi^{X_m}(a)|<\epsilon$.  
\end{lemma}

\begin{proof}
We induct on the complexity of $\varphi$.  First suppose that $\varphi(x)$ is the atomic formula $d(t_1(x),t_2(x))$.  Then \begin{multline*} |d(t_1^M(x),t_2^M(x))-d(t_1^{X_m}(x),t_2^{X_m}(x))| \leq \\ d(t_1^M(x),t_1^{X_m}(x))+d(t_2^M(x),t_2^{X_m}(x))<\epsilon \end{multline*} for $m$ sufficiently large by the previous lemma.  Now suppose that $\varphi(x)$ is the atomic formula $P(t_1(x),\ldots,t_k(x))$.  Then \begin{multline*} |P^M(t_1^M(x),\ldots,t_k^M(x))-P^{X_m}(t_1^{X_m}(x),\ldots,t_k^{X_m}(x))|= \\ |P^M(t_1^M(x),\ldots,t_k^M(x))-P^M(t_1^{X_m}(x),\ldots,t_k^{X_m}(x))|<\epsilon \end{multline*} if $m$ is sufficiently large so that $d(t_i^M(x),t_i^{X_m}(x))<\Delta_P(\epsilon)$ for each $i$.  Uniform continuity of connectives takes care of the connective case.

It remains to take care of the quantifier case.  Suppose that $\varphi(x)=\inf_y\psi(x,y)$.  Let $K'$ be as in the conclusion of the lemma for $\psi(x,y)$ and $\frac{\epsilon}{3}$.  Choose $K\geq K'$ so that $\frac{1}{K}\leq\Delta_\psi(\frac{\epsilon}{3})$.  We claim that this $K$ works.    

Suppose that $m\geq K$ and $a$ lies in $X_m$.  Let $r=\varphi^M(a)$ and $s=\varphi^{X_m}(a)$.  Let $b\in M$ be such that $\psi^M(a,b)< r+\frac{\epsilon}{3}$.  Let $c\in X_m$ be such that $d(b,c)\leq \frac{1}{m}$.  Then $\psi^M(a,c)\leq r+\frac{2\epsilon}{3}$.  By definition of $K'$, we obtain $\psi^{X_m}(a,c)< r+\epsilon$.  Consequently, $s< r+\epsilon$.  On the other hand, let $e\in X_m$ be such that $\psi^{X_m}(a,e)< s+\frac{\epsilon}{3}$.  Then by assumption, $\psi^M(a,e)< s+\frac{2\epsilon}{3}$, so $r< s+\epsilon$.  Thus, $|r-s|<\epsilon$.
\end{proof}

\begin{cor}
$M\equiv N$.
\end{cor}

\begin{proof}
By the above lemma, for any sentence $\sigma$, we have $\sigma^M=\lim_\mathcal{U}\sigma^{X_m}$.
\end{proof}

This proves Theorem~\ref{compfin} because then every $X_m$ will be a finite $L$-structure. To state the general result, we need the following concepts, which will be used also in the next section.

\begin{df}
An $L$-structure $Z$ is \emph{almost pseudofinite} (resp.\ \emph{almost pseudocompact}) if whenever $\sigma^A=0$ for an $L$-sentence $\sigma$ and all finite (resp.\ compact) almost $L$-structures $A$, then $\sigma^Z=0$.

Similarly, $Z$ is \emph{almost strongly pseudofinite} (resp.\ \emph{almost strongly pseudocompact}) if whenever $\sigma^Z=0$ for an $L$-sentence $\sigma$, then there is a finite (resp.\ compact) almost $L$-structure $A$ such that $\sigma^A=0$. 
\end{df}

(In all four cases, the ``almost'' version of a property is a consequence of the property itself.)

\begin{lemma} $Z$ is almost pseudofinite if and only if, whenever $\epsilon>0$ and $\sigma$ is an $L$-sentence such that $\sigma^Z=0$, then there is a finite almost $L$-structure $A$ such that $\sigma^A\leq \epsilon$.
\end{lemma}

\begin{proof}
This is analogous to the equivalence of (1) and (2) in Lemma~\ref{core-equiv}.
\end{proof}

Thus, since $M\equiv N$ and $N$ is an ultraproduct of finite almost $L$-structures, we obtain:

\begin{prop}
Every compact $L$-structure is \emph{almost} pseudofinite and thus \begin{center}
	pseudocompactness $\Rightarrow$ almost pseudofiniteness $\Rightarrow$ almost pseudocompactness.
\end{center}
\end{prop}

We conjecture that, in full generality, the notions of pseudofiniteness and pseudocompactness always agree. We will see in the next section (Example~\ref{circle}), however, that the strong versions are not equivalent.

\section{Injectivity-surjectivity of endofunctions}\label{sec-injsur}

The equivalence of injectivity and surjectivity of definable endofunctions is a direct consequence of pseudofiniteness in classical logic. Here, we experiment with the straightforward translation of that property in metric structures (although the ``right'' corresponding property may be very different of course).

Explicitly, in classical logic, if $M$ is pseudofinite and $f\colon M\to M$ is definable, then $f$ is injective if and only if $f$ is surjective.  The natural question is to ask whether this holds in the continuous setting.  Already in the simplest case this seems to require a few more assumptions.  Suppose that $M$ is a metric $L$-structure:

\begin{df}
Say that $f\colon M\to M$ is \emph{formula-definable} if there is an $L$-formula $\varphi(x,y,z)$, where $z$ is a tuple of variables, and there is a tuple $a$ from $M$ such that $d(f(x),y)=\varphi(x,y,a)$ for all $x,y\in M$.
\end{df}

\begin{prop}
Suppose that $M$ is $\omega$-saturated and strongly pseudofinite.  Suppose that $f\colon M\to M$ is a formula-definable function.  Then $f$ being injective implies that $f$ is surjective.
\end{prop}

\begin{proof}
Suppose that $f$ is injective but not surjective.  By $\omega$-saturation, there is $y\in M$ and an $\epsilon>0$ such that $d(f(x),y)\geq \epsilon$ for all $x\in M$. Fix $\varphi(x,y,z)$ and $a\in M$ such that $d(f(x),y)=\varphi(x,y,a)$ for all $x,y\in M$.  Note that, for all $x,y_1,y_2\in M$, we have \[ \max(\varphi(x,y_1,a),\varphi(x,y_2,a))=0\Rightarrow d(y_1,y_2)=0 .\]  Since $(M,a)$ is $\omega$-saturated, there is an increasing continuous function $\alpha\colon[0,1]\to [0,1]$ satisfying $\alpha(0)=0$ and such that, for all $x,y_1,y_2\in M$, we have \[ d(y_1,y_2)\leq \alpha(\max(\varphi(x,y_1,a),\varphi(x,y_2,a)) .\]  Similarly, since $f$ is injective, there is another increasing, continuous function $\beta\colon[0,1]\to [0,1]$ satisfying $\beta(0)=0$ and such that, for all $x_1,x_2,y\in M$, we have \[ d(x_1,x_2)\leq \beta(\max(\varphi(x_1,y,a),\varphi(x_2,y,a))) .\]  Consider the following formulae: \begin{align*}
	P(z) &:= \sup_x \inf_y \varphi(x,y,z) ,\\
	Q(z) &:= \sup_{x,y_1,y_2} [d(y_1,y_2)\dotminus \alpha(\max(\varphi(x,y_1,z),\varphi(x,y_2,z))] ,\\
	R(z) &:= \sup_{x_1,x_2,y} [d(x_1,x_2)\dotminus \beta(\max(\varphi(x_1,y,z),\varphi(x_2,y,z)))] ,\\
	S(z) &:= \inf_y\sup_x (\epsilon\dotminus \varphi(x,y,z)) .
\end{align*} Then \[ M\models \inf_z \max(P(z),Q(z),R(z),S(z))=0. \] Since $M$ is strongly pseudofinite, there is a finite $L$-structure $A$ such that \[ A\models \inf_z \max(P(z),Q(z),R(z),S(z))=0. \] Since $A$ is finite, $\inf$ quantifiers are actually realized, and thus there is $a_0\in A$ such that $\varphi(x,y,a_0)$ defines an injective function $A\to A$ which is not surjective, leading to a contradiction.
\end{proof}

\begin{prop}
Suppose that $M$ is $\omega$-saturated and strongly pseudofinite.  Suppose that $f\colon M\to M$ is a formula-definable function.  Then $f$ being surjective implies that $f$ is injective.
\end{prop}

\begin{proof}
Suppose that $f$ is surjective but not injective.  Define the formulae $P(z)$ and $Q(z)$ as in the previous proof.  Since $f$ is not injective, we can find $b_1,b_2\in M$ such that $d(b_1,b_2)=:\epsilon>0$ and $f(b_1)=f(b_2)$.  Consequently, there is an increasing continuous function $\gamma\colon[0,1]\to [0,1]$ satisfying $\gamma(0)=0$ and so that, for all $w_1,w_2\in M$, we have \[ d(w_1,w_2)\leq \gamma(\max(\varphi(b_1,w_1,a),\varphi(b_2,w_2,a))) .\]   Consider the formulae \begin{gather*} R(z):=\inf_{x_1,x_2}\max\bigl(|d(x_1,x_2)-\epsilon|,g(x_1,x_2,z)\bigr) \\ \text{where}\ g(x_1,x_2,z):=\sup_{w_1,w_2}[d(w_1,w_2)\dotminus \gamma(\max(\varphi(x_1,w_1,z),\varphi(x_2,w_2,z)))] ,\end{gather*} and \[ S(z):=\sup_y\inf_x \varphi(x,y,z) .\] Then \[ M\models \inf_z\max(P(z),Q(z),R(z),S(z))=0. \]  Thus this is true in a finite $L$-structure $A$, implying that there is a surjective definable function $A\to A$ which is not injective, contradiction.
\end{proof}

\begin{rmk}
The above propositions only required that $M$ was \emph{almost} strongly pseudofinite.
\end{rmk}

\begin{rmk}
A similar argument will prove directly that if $X$ is the zeroset of a formula in a power of $M$, and $f\colon X\to X$ is formula-definable (so there are $\varphi,a$ such that $d(f(x),y)=\varphi(x,y,a)$ for all $x,y\in X$), then $f$ is injective if and only if $f$ is surjective.
\end{rmk}

\begin{ex}\label{circle}
Let $\mathbb{S}^1=\{z\in\mathbb{C}\colon |z|=1\}$ have the metric which is half the one induced by the canonical metric in $\mathbb{C}$ (so it has values in $[0,1]$), and consider the ternary relation $P(u,v,w)=d(uv,w)$, where the usual product in $\mathbb{C}$ is used. Consider also $f\colon \mathbb{S}^1\to \mathbb{S}^1$, $f(z)=z^2$, which is surjective, but not injective. Then the relational structure $(\mathbb{S}^1,P)$ (in the minimal adequate language with the right modulus of uniformity for $P$) is compact, hence pseudofinite. Because of the total categoricity of compact models, $(\mathbb{S}^1,P)$ is saturated. Also, $f$ is formula-definable in $(\mathbb{S}^1,P)$: we have $d(f(z),w)=P(z,z,w)$. Therefore, by the results above, $(\mathbb{S}^1,P)$ is not strongly pseudofinite.
\end{ex}

\begin{ex}\label{interval}
Let $[0,1]$ have the usual metric, and consider $f\colon[0,1]\to[0,1]$, $f(x)=x/2$, which is injective, but not surjective. Similarly to the previous example, one obtains a compact, pseudofinite, but not strongly pseudofinite structure.
\end{ex}

Note that those examples show that the injective-surjective phenomenon does not necessarily hold for pseudofinite structures in the continuous sense, even for compact structures.

\begin{question}
Is there a natural property of endofunctions in continuous logic corresponding to injectivity-surjectivity, and which holds in pseudofinite structures? We confess our failure in detecting any such property; for example, the function $f\colon[0,1]\to[0,1]$ defined thus: $f(x)=2x$ if $0\leq x\leq\frac{1}{2}$ and $f(x)=1$ otherwise, is surjective, yet $f^{-1}(1)$ is \emph{huge} by any metric or topological standard.
\end{question}

It would be interesting to know if we could strengthen our injectivity-surjectivity results to hold for arbitrary definable functions in strongly pseudofinite structures.  Let us mention a few remarks towards proving this.  Suppose that $M$ is an $L$-structure and $P\colon M^n\to[0,1]$ is a definable predicate in $M$ (over some countable parameterset).  Let $L_P$ be the language obtained by adding a predicate symbol for $P$ and let $(M,P)$ be the natural expansion of $M$ to an $L_P$-structure.  Given an $L_P$-formula $\psi(y)$ without parameters and an $L$-formula $\varphi(x,a)$ with parameters $a$, where $|x|=n$, one naturally gets an $L(a)$-formula $\psi_\varphi$ by replacing every occurrence of $P(t)$ in $\psi$ with $\varphi(t,a)$ (for any occurring tuple $t$ of terms).

\begin{lemma}
If $M$ is $\omega_1$-saturated, then $(M,P)$ is $\omega_1$-saturated.
\end{lemma}

\begin{proof}
Suppose that $\{\psi^i(y)=0 \colon i\in I \}$ is a finitely satisfiable collection of $L_P$-conditions in countably many parameters.  By replacing $\psi^i(y)=0$ by $\psi_{\varphi^i_n}^i(y)\leq \frac{1}{n}$, where $\varphi^i_n$ is an $L$-formula approximating $P$ well enough, we obtain a finitely satisfiable collection of $L$-conditions in countably many parameters.  Then use $\omega_1$-saturation of $M$.
\end{proof}

\begin{prop}
Suppose that $f\colon M\to M$ is a definable function in an $\omega_1$-saturated structure $M$, let $P(x,y)=d(f(x),y)$, and suppose further that $(M,P)$ is almost strongly pseudofinite. Then $f$ is injective if and only if it is surjective.
\end{prop}

\begin{proof}
We have that $f$ becomes formula-definable in the almost strongly pseudofinite $\omega_1$-saturated structure $(M,P)$.
\end{proof}

\begin{question}
If $M$ is strongly pseudofinite, is $(M,P)$ almost strongly pseudofinite? 
\end{question}

Of course, if the answer to the above question is positive, then the extra assumption in the previous result is superfluous.  We were only able to settle the corresponding question for pseudofinite structures:

\begin{lemma}\label{replacement}
Given any $L_P$-formula $\psi(y)$ and $\epsilon>0$, there are parameters $a$ in $M$ and an $L(a)$-formula $\varphi(x,a)$ such that $|\psi^{(M,P)}(b)-\psi_\varphi^M(b)|\leq \epsilon$ for every $b\in M$.
\end{lemma}

\begin{proof}
Induct on the complexity of $\psi$.  First assume that $\psi$ is an atomic $L_P$-formula.  If $\psi$ is actually an $L$-formula, then there is nothing to do.  Otherwise, $\psi(y)$ is $P(t_1(y),\ldots,t_n(y))$, where $t_1,\ldots,t_n$ are $L$-terms; then choose $\varphi(x,a)$ such that $|P^M(x)-\varphi^M(x,a)|\leq \epsilon$ for all $x\in M^n$, so that $\varphi(t_1(y),\ldots,t_n(y),a)$ is the desired formula.  Connectives are handled as usual and the case of quantifiers is also immediate.
\end{proof}

\begin{lemma}
If $M$ is pseudofinite, then $(M,P)$ is almost pseudofinite.
\end{lemma}

\begin{proof}
Suppose that $\sigma$ is an $L_P$-sentence such that $\sigma^{(M,P)}=0$.  By Lemma~\ref{replacement}, given $\epsilon>0$, there is $\varphi(x,a)$ such that $(M,a)\models \sigma_\varphi\leq \frac{\epsilon}{2}$.  Thus, $M\models \inf_y \sigma_\varphi(y)\leq \frac{\epsilon}{2}$.  Assuming that $M$ is pseudofinite, we have that $A\models \inf_y\sigma_\varphi(y)\leq \epsilon$ for some finite $L$-structure $A$.  Let $b\in A$ be such that $(A,b)\models \sigma_\varphi(b)\leq \epsilon$.  Then make $A$ into an almost $L_P$-structure by interpreting $P(x)$ as $\varphi(x,b)$.  Then $(A,P^A)\models \sigma \leq \epsilon$.  
\end{proof}

\section*{Appendix: Compact structures}

As mentioned in the introduction, compact structures have no proper ultrapowers; indeed, since any sequence from a compact space has a unique ultralimit, the diagonal embedding ${M}\to {M}^\mathcal{U}$ of a compact structure $M$ into any of its ultrapowers is surjective.  From this it follows that ${M}$ is totally categorical, that is, if $N\equiv {M}$ then $N\cong M$.  To see this, we use the Keisler--Shelah theorem for continuous logic: if $N\equiv M$ then $N^\mathcal{U}\cong M^\mathcal{U}$ for some ultrafilter $\mathcal{U}$, hence $N^\mathcal{U}$ is compact.  By \L o\'s's theorem, we see that ${N}$ is compact, whence we have ${N}\cong {N}^\mathcal{U}\cong {M}^\mathcal{U}\cong {M}$.

Since we used the Keisler--Shelah theorem for continuous logic, we prefer to give a more elementary proof that compact structures are totally categorical.

\begin{prop}
Suppose that ${M}$ is a compact ${L}$-structure.  Then for any ${L}$-structure ${N}$, if ${N}\equiv {M}$, then ${N}\cong {M}$.
\end{prop}

\begin{proof}
Let $N\equiv M$. Without loss of generality, we may assume that ${N}$ is $\omega$-saturated.  Indeed, if $\mathcal{U}$ is a nonprincipal ultrafilter on $\mathbb{N}$, then ${N}^\mathcal{U}\equiv {M}$ and ${N}^\mathcal{U}$ is $\omega_1$-saturated.  If further ${N}^\mathcal{U}\cong {M}$, then again ${N}\cong {M}$ by \L o\'s's theorem.

For ease of notation, we assume that ${L}=\{P,F\}$, where $P$ is a unary predicate symbol and $F$ is a unary function symbol.  The proof that we give below extends immediately to arbitrary finite languages.  For arbitrary languages, one needs to replace single conditions by partial types.
 
For $m\geq 1$, let $\{a_1^m,\ldots,a^m_{n(m)}\}$ be a finite $\frac{1}{m}$-net for $M$.  For $i\in\{1,\ldots,n(m)\}$, let $r(i,m):=P^M(a_i^m)$ and fix $j(i,m)\in \{1,\ldots,n(m)\}$ such that \[ d(F(a_i^m),a^m_{j(i,m)})\leq \tfrac{1}{m}. \]  For $i,j\in \{1,\ldots,n(m)\}$, set $s(i,j,m):=d(a_i^m,a_j^m)$.  Consider the following formulae: \begin{align*}
	\psi_m(x_1,\ldots,x_{n(m)}) &:= \sup_x\bigl(\min_{1\leq i\leq n(m)}(d(x,x_i)\dotminus \tfrac{1}{m})\bigr),\\
	\chi_m(x_1,\ldots,x_{n(m)}) &:= \max_{1\leq i\leq n(m)}\bigl(\max\bigl(|P(x_i)-r(i,m)|,d(F(x_i),x_{j(i,m)})\dotminus \tfrac{1}{m}\bigr)\bigr),\\
	\tau_m(x_1,\ldots,x_{n(m)}) &:= \max_{1\leq i,j\leq n(m)}|d(x_i,x_j)-s(i,j,m)|,\\
	\varphi_m(x_1,\ldots,x_{n(m)}) &:= \max(\psi_m,\chi_m,\tau_m).
\end{align*}

Since ${N}\equiv {M}$ and ${N}$ is $\omega$-saturated, we have that there exists, for each $m\geq 1$, $b_1^m,\ldots,b_{n(m)}^m\in N$ such that $\varphi_m^N(b_1,\ldots,b_{n(m)})=0$.  Set $A_m:=\{a_1^m,\ldots,a_{n(m)}^m\}$ and $B_m:=\{b_1^m,\ldots,b^m_{n(m)}\}$. It remains to observe that $A:=\bigcup_m A_m$ is dense in $M$ and $B:=\bigcup_m B_m$ is dense in $N$.
\end{proof}

Our proof shows that, unlike finite structures in finite languages in classical logic, compact structures in a finite language are not finitely axiomatizable, but rather, countably axiomatizable.

\end{document}